\documentclass[10pt]{amsart}
\usepackage[centertags]{amsmath}
\usepackage{fullpage}
\usepackage{amsfonts}
\usepackage{amsthm}
\usepackage{newlfont}
\usepackage{amscd}
\usepackage{amsgen}
\usepackage{amssymb}
\usepackage{longtable}
\usepackage{listings}
\newlength{\defbaselineskip} 

 \theoremstyle{plain} \newtheorem{thm}{Theorem}[section]
\newtheorem{pro}[thm]{Problem} \newtheorem{lemm}[thm]{Lemma}
\newtheorem{prop}[thm]{Proposition}
\theoremstyle{definition}

\newtheorem{rem}[thm]{Remark}
\newtheorem{obs}[thm]{Observation}
\newtheorem{cor}[thm]{Corollary}
 \numberwithin{equation}{section}
\numberwithin{equation}{section}
\theoremstyle{definition} 
\newtheorem*{Claim}{Claim}

\title{Projections of Mukai Varieties}
\author{Micha\l \ Kapustka}
\thanks{This project grew up while the author visited the
University of Oslo supported by an EEA Scholarship and Training fund
in Poland and by MNiSW grant N N201 388834}
\begin{document}

\begin{abstract}
This note is an answer to a problem proposed by Iliev and Ranestad. We prove
that the projections of general nodal
linear sections of suitable dimension of Mukai varieties $M_g$ are linear
sections of $M_{g-1}$.
\end{abstract}
\maketitle
\section{Introduction}
In \cite{Muk1} Mukai gave a description of general canonical curves, K3 surfaces
and Fano threefolds of sectional genus $g\leq 10$ in terms
of linear sections of appropriate varieties. For prime Fano threefolds of index 1
the description gives a full classification up to genus $g\leq 10$. It may be summarized
in the following table
\begin{longtable}{|c|c|}
\hline
 genus & Anti-canonical model of a prime Fano \\
&threefold of index $1$ and given genus\\
\hline
$2$& $X_6\subset \mathbb{P}(1^4,3)$\\
\hline
$3$& $X_4\subset \mathbb{P}^4$ \\
\hline
$4$& $X_{2,3}\subset \mathbb{P}^5$\\
\hline
$5$& $X_{2,2,2}\subset \mathbb{P}^6$\\
\hline
$6$& $X_{1,1}\subset Q_2\cap G(2,5)=: M_6^{5}$\\
\hline
$7$& $X_{1,1,1,1,1,1,1}\subset OG(5,10)=:M_7^{10}$\\
\hline
$8$& $X_{1,1,1,1,1}\subset G(2,6)=: M_8^8$\\
\hline
$9$& $X_{1,1,1}\subset LG(3,6)=: M_9^6$\\
\hline
$10$ & $X_{1,1} \subset G_2=: M_{10}^5$\\
\hline
\end{longtable}
In the table we use the notation $X_{i_1,..,i_n}$ for the generic complete
intersection of given degrees.
The variety $Q_2$ is a generic quadric hypersurface. The notation $G(2,n)$ stands
for the Grassmannians of lines in projective $n-1$-space
in their Pl\"ucker embeddings. The variety $OG(5,10)$ is the orthogonal
Grassmannian. It is a component of the set of linear spaces
of dimension $4$ contained in a smooth eight dimensional quadric hypersurface in
$\mathbb{P}^9$ in its spinor embedding. The
variety $LG(3,6)$ is the Lagrangian Grassmannian, it is a linear section of
$G(3,6)$ in its Pl\"ucker embedding
parametrizing 3-dimensional vector spaces isotropic with respect to a chosen generic symplectic
form. The variety $G_2$ is a linear section of $G(5,7)$ in its Pl\"ucker embedding parametrizing 5-dimensional vector subspaces
of a 7-dimensional vector space isotropic with respect to a chosen
generic four-form. The notation $M_g$ and the name Mukai varieties has become
common in this context. The upper index used in the table stands for the dimension of the variety and will be omitted.
We shall describe these varieties more precisely in section \ref{mukai varieties}. 

In fact there is only one more family of prime Fano threefolds of index 1. It corresponds to the case $g=12$.

It is now a natural problem to relate these Fano varieties by means of standard constructions such for example as projections.
In particular the following problem was considered in \cite{R,RI}.
\begin{pro}
For given $7 \leq g\leq 10$, what is the highest $n$ such that there exists a proper linear section $H$ of dimension $n$ of $M_g$ admitting 
a single ordinary double point $p$ as singularity and such that the projection of $H$ from $p$ is linearly isomorphic to a proper linear section of $M_{g-1}$. 
\end{pro}
The justification to proposing this problem is the observation that taking the projection of a nodal Fano manifold (K3 surface or canonical curve) of sectional genus $g$ from the node 
we still get a Fano manifold (K3 surface or canonical curve) but with sectional genus reduced by 1, hence the result should appear as a section of $M_{g-1}$. The only problem arising is that the resulting 
variety might again be (and in fact will always be) singular in which case Mukai's result does not work. 

As evidence in \cite{R} it was observed that the statement is true for $n=1$. Moreover an upper bound for $n$ was given, by computing the maximal dimension 
of quadrics contained in $M_{g-1}$ and observing that the result of the considered projection must contain a quadric divisor as the exceptional divisor of the projection.   

Observe moreover that $n$ can be arbitrarily large for an analogous problem formulated for $g\leq 5$.
More precisely we have the following observation:
\begin{obs}
For $2\leq g\leq 5$ and $n\in \mathbb{N}$ there exists a complete intersection $M$ of type $M_g$ (i.e. as in the Table) in the corresponding weighted projective space 
such that $M$ admits a single ordinary double point as singularity. Moreover for any such $M$ the projection from the node is linearly isomorphic 
to a complete intersection of type $M_{g-1}$. Conversely a generic complete intersection of type $M_{g-1}$ containing a smooth quadric as a codimension 1 subvariety
can be obtained in such a way. 
\end{obs}
Similarly for $g=6$.
\begin{obs}
There exists a quadric $Q$ such that $G(2,5)\cap Q$ has a single node.  Moreover for any such intersection the projection from the node is linearly isomorphic 
to a complete intersection of type $X_{2,2,2}$. Conversely a generic complete intersection $X_{2,2,2}$ containing a smooth quadric as a codimension 1 subvariety
can be obtained in such a way. 
\end{obs}

Indeed these are examples of standard Kustin Miller unprojections.

The case $g=9$ was solved in \cite{RI}. Before we state the theorem let us observe that the general singular hyperplane section of $LG(3,6)$ has a single node as singularity. 
Let now $L$ be any nodal hyperplane section of $LG(3,6)$ and $p$ its unique singularity.
\begin{thm}\label{z LG(3,6) w G(2,6)} The projection of $L$ from the node $p$ is a proper codimension 3 linear section of
$G(2,6)$, containing a 4 dimensional quadric. Conversely a general 5
dimensional linear section of $G(2,6)$ that contains a 4 dimensional
quadric arises in this way.
\end{thm}
The proof followed from the construction of an appropriate bundle on the
resolution of a nodal hyperplane section
of $LG(3,6)$.

In this note we reprove Theorem \ref{z LG(3,6) w G(2,6)} together with the remaining
cases in purely algebraic terms by analysis of equations of considered varieties in terms of natural representations appearing on the linear spaces they span.

The original motivation of \cite{R,RI} for studying the problem was the construction of non-abelian Brill-Noether loci   
in moduli spaces of bundles over Mukai varieties. We pursue this idea in section \ref{modBNL}.

Our main focus however will be put on the understanding of the geometry of the constructions presented with a view toward future applications in
the theory of Mirror Symmetry and Landau-Ginzburg models. For this reason in section \ref{mir} we concentrate on the case of Fano 3-folds.
We prove that for a Fano 3-fold of genus $g$ admitting a single node its projection form the node is a Fano 3-fold of genus $g-1$ with also only 
nodes as singularities. We factorize the projection into a blow up of the node and a small contraction of lines and count the number of nodes obtained 
in each case. In this way we connect families of Fano 3-folds of genus $g$ in the simplest way from the point of view of the theory of Landau-Ginzburg models.

The analogue of this in the case of Calabi-Yau threefolds is a cascade of geometric bitransitions connecting Calabi-Yau threefolds from the list of Borcea (see \cite{Unpr}).

\begin{section}{Statements}
The main results of the paper may be summarized as follows 
\begin{thm}\label{z S10 w G(2,5)}
The subcheme of $G(10,15)$ parametrizing singular linear sections of $M_{7}$ is
irreducible.
The general element of this subcheme corresponds to a 5-dimensional linear
section $L$ of $M_{7}$ admitting a single node. 
The projection of $L$ from the node is isomorphic to a proper intersection $G(2,5)\cap Q$, where $Q$
is a quadric in $\mathbb{P}^9$ such that $G(2,5)\cap Q$ contains a 4 dimensional quadric.
Moreover a generic variety $G(2,5)\cap Q'$ containing a 4 dimensional quadric arises in this way.
\end{thm}

\begin{thm}\label{z G(2,6) w S10}
The projective dual variety of $M_8=G(2,6)$ is irreducible. The general element of this
variety defines a hyperplane section $L$ of $G(2,6)$ admitting a single node as singularity. 
The projection of $L$ from the node is then a proper linear section of
$OG(5,10)$, containing a 6 dimensional quadric. Moreover a generic linear section of 
$OG(5,10)$ containing a 6 dimensional quadric arises in this way. 
\end{thm}

\begin{thm}\label{z G2 w LG(3,6)} The projective dual variety to $G_2$ is irreducible.  
The general element of this
variety defines a hyperplane section $L$ of $G_2$ admitting a single node as singularity.  
Let $L$ be any hyperplane section of $G_2$ admitting a single node. Then
the projection of $L$ from the node is a proper linear section of,
$LG(3,6)$, containing a 3 dimensional quadric. Moreover a generic linear section of $LG(3,6)$ containing 
a 6 dimensional quadric arises in this way.
\end{thm}
The proof of Theorems \ref{z S10 w G(2,5)}, \ref{z G(2,6) w S10}, \ref{z LG(3,6)
w G(2,6)} and \ref{z G2 w LG(3,6)}
is based on analyzing representations appearing on the linear sections involved and comparing the equations of the varieties involved.

\section{Descriptions of Mukai varieties, their tangents and projective duals} \label{mukai varieties} In this section we recall the known descriptions of Mukai
varieties and their projective duals. Moreover we provide a classification of maximal dimensional quadrics contained in them. 
As reference for the descriptions contained in this section we suggest \cite{Muk1,MukGrass,RS,SK}.

We start with the general description of the Grassmannian $G(2,n)$.
\subsection{The Grassmannian $G(2,n)$}  Let $V$ be a $n$-dimensional vector space with $n\geq 2$. The Grassmannian $G(2,V)$ 
is then the subvariety of $\mathbb{P}(\bigwedge^2 V)$ consisting of simple forms. It is scheme theoretically the zero locus 
of the quadratic form:
$$sq_V:\wedge^2 V \ni \omega\mapsto \omega\wedge\omega \wedge^4 V.$$
The Grassmannian is also a homogeneous space of $\operatorname{GL}(V)$. In this language if $V$ is the standard representation of $GL(V)$
then $G(2,V)$ is the unique closed orbit of the projectivized representation $\mathbb{P}(\bigwedge^2 V)$.
Let us now fix a point $p$ in $G(2,V)$ i.e. a two-dimensional subspace $V_2 \in V$. The stabilizer subgroup of $p$ is the parabolic subgroup 
$P$ of $GL(V)$ consisting of automorphism preserving $V_2$. By standard Lie theory $P$ has a decomposition into a semi-direct product of a 
semi-simple Lie group and a solvable ideal. Such a semi-simple Lie group is called a Levi subgroup of $P$. It is also known that all Levi subgroups are conjugate.
 In our case a choice of Levi subgroup of $P$ corresponds to a choice of decomposition $V=V_2\oplus V_4$, then the Levi subgroup is the direct product 
$GL(V_2)\times GL(V_4)$.
The representation $\mathbb{P}(\bigwedge^2 V)$ restricted to the Levi subgroup decomposes into 
$$\mathbb{P}(\bigwedge^2 V_2\oplus (V_2\otimes V_4) \oplus \bigwedge^2 V_4).$$
Let us consider the quadratic form $sq_V$ with respect to the above decomposition.
To do this we first observe that $\wedge^4 V$ restricted to our Levi subgroup also decomposes:
$$\wedge^4 V=(\bigwedge^2 V_2 \otimes \bigwedge^2 V_4)\oplus (V_2\otimes \bigwedge^3 V_4) \oplus \det(V_4).$$
Now $$sq_V\colon (\omega_2, \varphi, \omega_4)\mapsto (\varphi\wedge\varphi+\omega_2\otimes \omega_4,\varphi \wedge \omega_4 ,\omega_4\wedge\omega_4).$$ 
It follows that the invariant subspace $\mathbb{P}(\bigwedge^2 V_2 \oplus (V_2\otimes V_4))$ is the tangent subspace of the Grassmannian $G(2,V)$ in the point $p$.

The projective dual variety of the Grassmannian $G(2,V)$ is described as the zero locus of the symmetric $[\frac{n}{2}]$-form
$$pf^n_V\colon \bigwedge^2 V\ni \omega \mapsto \omega\wedge \dots \wedge\omega \in \bigwedge^{2[\frac{n}{2}]} V.$$
It follows that the projective dual variety of $G(2,V)$ is a hypersurface or is of codimension 3. In both cases it is irreducible. 
       
For the study of maximal dimensional quadrics in the Grassmannian $G(2,5)$ and $G(2,6)$ it is well known that these quadrics are of dimension 4 and are described
as the set of 2-spaces contained in a chosen 4-dimensional vector subspace of the vector space defining the Grassmannian.  

\subsection{The Orthogonal Grassmannian $OG(5,10)$} Let us start with some generalities about the variety $OG(n,2n)$. To define this space we 
start with a $2n$-dimensional vector space $V_n$ endowed with a quadratic form $q$. Consider the variety $S$ of 5-subspaces of $V_{10}$ isotropic with 
respect to $q$. It is a subvariety of the Grassmannian $G(n,2n)$ having two components $S^{ev}$ and $S^{odd}$. 
They are called the even and odd orthogonal Grassmannians and are denoted by $OG(n,2n)$.
The orthogonal Grassmannian is a homogeneous variety of the group $SO_q(V)$.
In this paper we are interested in the so-called spinor embeddings of these varieties.
A convenient way to get the description of the image of this embedding is to start with a point $p\in OG(n,2n)$ i.e. a subspace $V_n$
isotropic with respect to $q$. We then have a natural decomposition $V=V_n\oplus V_n^*$ in which $q$ is given by the matrix 
$$\left(\begin{array}{cc}
   0 &I_n\\
I_n & 0
  \end{array}\right).
 $$ 
The parabolic subgroup of $SO_q(V_{2n})$ of elements preserving $V_n$ has as Levi subgroup $GL(V_n)$. It is however more convenient to write 
the spinor embedding as an invariant variety in terms of $SL(V_n)$ representation:
$$\mathbb{P}(\bigwedge^{ev} V_n),$$
the projectivization of the even part of the exterior algebra of the standard representation $V_n$.  
The even orthogonal Grassmannian in its spinor embedding is then described in this space as the closure of the image of the exponential map:
$$ exp: \bigwedge^2 V_n\ni \omega \longrightarrow 1+\sum_{i=1}^n \frac{1}{n!}\omega^{\wedge i} \in \bigwedge^{ev} V_n.$$

In our case $n=5$ and $$\mathbb{P}(\bigwedge^{ev} V_5)=\mathbb{C}\oplus \bigwedge^2 V_5 \oplus \bigwedge^4 V_5.$$
To get a set of equations in an intrinsic way we use the identification of $SL(V_5)$ representations $\mathbb{C}=det V_5$ and $\bigwedge^4 V_5=V_5^*$.
The orthogonal Grassmannian is now scheme theoretically the zero locus of the quadratic form:
$$\det{V_5}\oplus \bigwedge^2 V_5 \oplus V_5^*\ni (x,A,v)\mapsto (x(v)+A\wedge A, A(v)) \in \bigwedge^4 V_5 \oplus V_5.$$
Finally observe that the above form is invariant with respect to the $GL(V_5)$ action on $\det{V_5}\oplus \bigwedge^2 V_5 \oplus V_5^*$, hence $OG(5,10)$ 
is a $GL(V_5)$ invariant subvariety in  $\mathbb{P}(\det{V_5}\oplus \bigwedge^2 V_5 \oplus V_5^*)$. In fact one checks easily that it is the closure of one orbit.

The tangent space to $OG(5,10)$ attached to the point $p=\mathbb{P}(\det{V_5})$ is clearly the space $\mathbb{P}(\det{V_5}\oplus \bigwedge^2 V_5)$.
Moreover it is a well known theorem (see for example \cite{Ein}) that the variety $OG(5,V_{10})$ in its spinor embedding is self dual. More precisely its dual variety is 
$OG(5,V_{10}^*)$ embedded via its spinor embedding in $\mathbb{P}(\det{V_5^*}\oplus \bigwedge^2 V_5^* \oplus V_5)=\mathbb{P}(\det{V_5}\oplus \bigwedge^2 V_5 \oplus V_5^*)^*$.

We recall also the classification of maximal dimensional quadrics in $OG(5,10)$ due to Ranestad. By \cite{RS} the quadrics defining the orthogonal Grassmannian 
$OG(5,10)$ define a rational map $v^+$ from $\mathbb{P}^{15}$ onto a quadric in $\mathbb{P}^9$.
We claim that the maximal dimensional quadrics are the intersections of the $\mathbb{P}^7$ fibers of this map with the orthogonal Grassmannian  $OG(5,10)$. 
Indeed by the interpretation in \cite{RS} the fibers are projective spaces of dimension 7. To see that they intersect the Grassmannian in a quadric it is enough to observe that 
the quadrics defining the Grassmannian $OG(5,10)$ restricted to any chosen fiber of $v^+$ by definition form a one dimensional vector space. They hence define a quadric hypersurface on the $\mathbb{P}^7$.
For the converse let $Q$ be a quadric of dimension $\geq 6$ in $OG(5,10)$. Then $Q$ spans a linear space that is not contained in $OG(5,10)$ and meets it in a quadric hypersurface. 
It follows that this space must be contained in a fiber of $v^+$ hence must be a quadric of dimension 6 described above.

\subsection{The Lagrangian Grassmannian $LG(3,6)$}\label{eq Lagr}
For a chosen vector space $V_{2n}$ of dimension $2n$ and a generic symplectic form $\omega\in \bigwedge^2 V_{2n}$
the variety $LG_{\omega}(n,V_{2n})$ is the subvariety of the Grassmannian $G(n, V_{2n})$ parametrizing $n$-spaces isotropic with respect to the form $\omega$.
In this way $LG_{\omega}(n,V_{2n})$ is a non-proper linear section of the Grassmannian $G(n, V_{2n})$.
The embedding that we consider is the one coming from the Pl\"ucker embedding of the Grassmannian. 
The variety $LG(n,2n)$ is a homogeneous variety of the simple Lie group $Sp_{\omega}(V_{2n})$ of automorphisms of $V_{2n}$ preserving the form $\omega$.

From now on to avoid technicalities we concentrate on the case $n=3.$
As in the previous case to get a suitable description of our variety it is convenient to fix a point $p\in LG(3,V_6)$
 i.e. a subspace $V_3$ isotropic with respect to $\omega$.
Again  we get in this way a natural decomposition $V_{6}=V_3 \oplus V_3^*$ such that omega is given by the matrix:
$$\left(\begin{array}{cc}
   0 &I_3\\
-I_3 & 0
  \end{array}\right).
 $$ 
Then $\bigwedge^3 V_{6}=\det V_3 \oplus (\bigwedge^2 V_3 \otimes V_3^*) \oplus (\bigwedge^2 V_3^* \otimes V_3) \oplus \det V_3^* $.
Now we observe that:
$\bigwedge^2 V_3 \otimes V_3^*=(S^2 V_3 \otimes \det V_3^*) \oplus V_3$
and the span of the Lagrangian Grassmannian is the subspace:
$$\det V_3 \oplus (S^2 V_3 \otimes \det V_3^*) \oplus (S^2 V_3^* \otimes \det V_3) \oplus \det V_3^* .$$

Before we pass to the equations describing $LG(3,V_6)$ let us introduce some notation. 
As usual the evaluation map will be denoted by 
$$\det(V_3)\otimes \det(V_3)^*\ni a\otimes b \mapsto a(b)=b(a)\in \mathbb{C}.$$
as well as any map based on this evaluation as for instance:
$$(S^2 V_3 \otimes \det(V_3)^*)\otimes \det(V_3) \ni B\otimes a  \mapsto B(a)\in S^2 V_3$$
and
$$(S^2 V_3^* \otimes \det(V_3))\otimes \det(V_3)\ni A\otimes b \mapsto A(b)\in S^2 V_3^*.$$
We moreover have the natural projection
$$S^2(S^2 V_3 \otimes \det V_3^*)=(S^4 V_3\otimes (\det V_3^*)^2) \oplus S^2 V_3^* \xrightarrow{\pi}  S^2 V_3^*$$
and on the dual space
$$S^2(S^2 V_3^*\otimes \det V_3)=(S^4 V_3^*\otimes (\det V_3)^2) \oplus S^2 V_3 \xrightarrow{\pi'}  S^2 V_3^*$$
Finally we have two projections from
$$(S^2 V_3 \otimes \det V_3^*) \otimes (S^2 V_3^* \otimes \det V_3)=S^2 V_3 \otimes S^2 V_3^*= \Phi_{2,2,0}\oplus \Phi_{1,1,0} \oplus \mathbb{C}$$
onto 
$\Phi_{1,1,0}$ and  $\mathbb{C}$ which we shall denote by $\eta_1$ and $\eta_2$ respectively. Here we follow the standard notation of \cite[\S 15.5]{FH} for the representation $\Phi_{i,j,k}$.  

With the above notation the Lagrangian Grassmannian $LG(3,V_6)$ is defined as the zero locus of the form
\begin{displaymath}
\begin{split}
 & \det V_3 \oplus (S^2 V^*_3 \otimes \det V_3) \oplus (S^2 V_3 \otimes \det V_3^*) \oplus \det V_3^*\ni(a,A,B,b)\mapsto\\ 
 &(\eta_1(A\otimes B), b(a)- \eta_2(A\otimes B), \pi(A)-B(a),\pi'(B)-A(b))\in \\
&\Phi_{1,1,0}\oplus \mathbb{C} \oplus S^2V_3^* \oplus S^2 V_3 
\end{split}
\end{displaymath}

The attached tangent space to $LG(3,V_6)$ in the point $p$ is the space 
$$\det V_3 \oplus (S^2 V^*_3 \otimes \det V_3).$$ 
Finally the projective dual variety to $LG(3,V_6)$ is an irreducible quartic hypersurface. For a more detailed description of the quartic and the type of 
singularities corresponding to orbits in its stratification we send the reader to \cite{RI}. We shall use the fact that there is a unique orbit giving 
nodal sections, and it is the open orbit of the quartic.   

Finally we recall the description of maximal dimensional quadrics contained in $LG(3,6)$ due to Ranestad.
First observe that any conic on $G(3,6)$ which is not contained in a plane contained in $G(3,6)$
parametrizes planes contained in a quadric of rank 5. Now, on a quadric, through any 3 points there is a conic passing through them.
It follows that every quadric contained in $G(3,6)$ is contained in some $G(1,3,6)$ denoting the subvariety of the Grassmannian $G(3,6)$ consisting of planes passing through a fixed point. 
But the intersection $G(l,3,V_6)\cap LG_{\omega}(3,V_6)$ for any
one-dimensional subspace $l\subset V_6$ is a quadric of dimension 3. It follows that the maximal dimensional quadrics in $LG(3,6)$ are of dimension 3 and obtained as
intersections  $G(l,3,V_6)\cap LG_{\omega}(3,V_6)$.

\subsection{The adjoint $G_2$ variety} \label{eqG2}

Let $V_7$ be a vector space of dimension $7$ understood as a standard representation under the action of the group $GL(V_7)$. 
Then the representation $\bigwedge^4 V_7^*$ admits an open orbit. Choose a 4-form $\omega \in \bigwedge^4 V_7^*$ 
from this open orbit. The variety $G_2$ is the subvariety of the Grassmannian $G(2,V_7)$,
consisting of those two-spaces which are isotropic with respect to $\omega$. To see it as a homogeneous space observe that the stabilizer subgroup of 
$\omega$ in the representation $\bigwedge^4 V_7^*$ is a simple Lie group called $\mathbb{G}_2$. The representation of $\mathbb{G}_2$ on $V_7$ is irreducible and 
called the standard representation of $\mathbb{G}_2$. We shall denote it $V_7^{\mathbb{G}_2}$. Now $\bigwedge^2 V_7^{\mathbb{G}_2}$ decomposes into 
$V_7^{\mathbb{G}_2}\oplus \operatorname{Ad}_{\mathbb{G}_2}$, where $\operatorname{Ad}$ denotes the adjoint representation of the group $\mathbb{G}_2$. In this case the 
space $\operatorname{Ad}_{\mathbb{G}_2} $ is the subspace of $\wedge^2 V_7$ consisting of elements isotropic with respect to $\omega$. The variety $G_2$ is
therefore obtained as the intersection $\mathbb{P}(\operatorname{Ad}_{\mathbb{G}_2})\cap G(2,V_7)$ and thus is the unique closed orbit of 
the projectivized adjoint representation of $\mathbb{G}_2$. In particular $G_2$ is a homogeneous space.    

For an intrinsic way to get the equations let us fix a point $p\in G_2$ i.e. a subspace $V_2$ of dimension 2 isotropic with respect to $\omega$. The stabilizer  
subgroup of $p\in G_2$ contains $SL(2)$ embedded in such a way that $V_7^{G_2}=V_2\oplus V_2^* \oplus (S^2V_2\otimes \det V_2^*)$. Then after restriction we have 
\begin{displaymath}
\begin{split}
\bigwedge^2 V_7^{G_2}=V_2\oplus V_2^* \oplus (S^2V_2\otimes \det V_2^*)\oplus \det V_2 \oplus (S^3 V_2\otimes \det V_2^*) \oplus (S^2V_2\otimes \det V_2^*)\\
\oplus \mathbb{C} \oplus 
(S^3 V_2^*\otimes \det V_2)\oplus \det V_2^*,
\end{split}
\end{displaymath}
and 

$$\operatorname{Ad}_{\mathbb{G}_2}=\det V_2 \oplus (S^3 V_2\otimes \det V_2^*) \oplus (S^2V_2\otimes\det V_2^*)\oplus \mathbb{C} \oplus 
(S^3 V_2^*\otimes \det V_2)\oplus \det V_2^*.$$
In fact $G_2$ and $p$ are invariant under the projectivization of the above action understood as a $GL(2)$ representation.  

By the description above, the variety $G_2$ being the intersection $\mathbb{P}(\operatorname{Ad}_{\mathbb{G}_2})\cap G(2,V_7)$ is described as the zero locus of the form:
$\operatorname{Ad}_{\mathbb{G}_2}\ni A\mapsto A\wedge A \in \wedge^4 V_7^{\mathbb{G}_2}$

The following Lemma provides us a classification of orbits of hyperplanes in the projectivization of the adjoint representation $\operatorname{Ad}(\mathbb{G}_2)$ giving rise to singular sections 
of $G_2$. Recall that in \cite[lemma 1]{M3l3} a classification of all orbits of the co-adjoint representation lying outside the dual variety of 
$G_2$ is given in terms of a family of sextic hypersurfaces.
In the lemma below we complete this classification with known results concerning orbits contained in the dual variety.  
\begin{lemm}\label{orbity}
The discriminant variety of the adjoint representation of the simple Lie group $\mathbb{G}_2$ admits the following orbits:
\begin{itemize}
 \item An open orbit $O_{12}$ of dimension 12 
\item One orbit $O_{11}$ of dimension 11 being an open subset of the base locus of sextic hypersurfaces
\item One orbit $O_{10}$ of dimension 10 being an open subset of the singular locus of the discriminant sextic.
\item One orbit $O_9$ in dimension 9 being an open subset of the intersection of $\overline{O_{11}}\cap \overline{O_{10}}$.
\item One orbit $O_7$ of dimension 7.
\item One orbit $O_5$ of dimension 5 corresponding to the variety $G_2$ in the co-adjoint representation being isomorphic to the adjoint representation by 
the Killing form.   
\end{itemize}
\end{lemm}
\begin{proof} The orbits $O_{11},O_{9},O_{7},O_{5}$ are the projectivizations of the well known 4 nontrivial nilpotent orbits of the semi-simple Lie group $\mathbb{G}_2$.
The orbit $O_{12}$ is the semi-simple orbit corresponding to long root vectors of Cartan sub-algebras. The orbit $O_{10}$ is the unique orbit of mixed type. The proof of uniqueness 
and its description is completely analogous to the description of the mixed orbit with semi-simple part being a short root vector of a Cartan sub-algebra contained in \cite[lemma 1]{M3l3}.   
\end{proof}

\end{section}
\begin{section}{The proofs}
The idea of the proofs is the following. We start with the fact that each of our Mukai varieties 
appears as an orbit of a representation of a suitable Lie Group on a projective space. 
Now on one hand we consider the representation corresponding to $M_{g}$ restricted to a suitable subgroup preserving a singular hyperplane sections
on the other we have the representation corresponding to $M_{g-1}$ restricted to a subgroup preserving a linear space of suitable dimension containing
a quadric.
Finally we identify these restricted representations and the invariant varieties corresponding to considered varieties.

\begin{proof}[Proof of theorem \ref{z S10 w G(2,5)}]
Let us consider a point $p\in OG(5,10)$ the isotropy group of $p$ contains
$GL(5)$ as a Levi subgroup. The corresponding action of $ GL(5)$ onto
$\mathbb{P}^{15}$ decomposes into:
$$\mathbb{P}(\det V \oplus \bigwedge^2 V \oplus V^{*}),$$
where $V$ is the standard representation of $GL(5)$. Let us denote the coordinates corresponding to this decomposition
by $(x,A,v)$. In these coordinates $OG(5,10)$ is described as the zero locus of $x(v)+ A\wedge A=0$ and $A(v)=0$.
As $p$ is given by $A=0$ and $v=0$, the tangent space to $OG(5,10)$ in $p$ is given by $v=0$. A hyperplane containing the tangent space of $OG(5,10)$ in the
fixed point $p$ of the group action is then given by choosing a hyperplane $U$ in $V$. A Levi subgroup of the stabilizer of such a hyperplane is isomorphic to 
$GL(4)$ and corresponds  to a decomposition of $V$ into the $GL(4)$ representation $U\oplus \mathbb{C}$ where by abuse of notation $U$ is now the standard representation 
of $GL(4)$. The  $GL(4)$ representation on the tangent hyperplane is thus $\mathbb{P}(\det U \oplus \bigwedge^2 U \oplus
U \oplus U^*)$, hence on the projection from $p$ we have the representation
$\mathbb{P}(\bigwedge^2 U \oplus U \oplus U^*)$. Denote the coordinates corresponding to this decomposition by 
$(B,u^*,u)$. 
\begin{Claim}
 The equations describing the projection are $B\wedge B=0$, $B(u)=0$ and $u^*(u)=0$.
\end{Claim}
Indeed let us denote the vector corresponding to the fixed part of the $GL(4)$ action on $V$ by $t$ and its corresponding dual by $t^*$. So that 
$A= B+t^*\wedge u^*$, $v^*=u^*+ t^*$, $v=u+t$ and $x=x'\wedge t^*$. The equations of $OG(5,10)$ in these coordinates are:
$$(x'\wedge t^*)(u+t) + (B+t^*\wedge u^*)\wedge(B+t^*\wedge u^*)=0, \quad (B+t^*\wedge u^*)(u+t)=0.$$
Expanding these we get:
$$(-x'(u)+2 B\wedge u^*)\wedge t^* + x'\wedge t^*(t)+B\wedge B=0, \quad B(u)-t^*(t) \wedge u^* + t^*\wedge u^*(u)=0.$$
We now decompose the equations with respect to $t^*$, put $t=0$, and consider all equations not involving $x'$ and obtain the claim.

Consider now on the other hand the following $GL(4)$ representation 
$$\mathbb{P}((\bigwedge^2 (U^*\oplus \mathbb{C})) \oplus U)=\mathbb{P}(\bigwedge^2 U^* \oplus U^* \oplus U).$$ 
Denote the coordinates corresponding to the above decomposition by $(B',w'^*,w')$. It clearly contains as an invariant
subvariety the cone $G$ over the Grassmannian $G(2,U^*\oplus \mathbb{C})$ and the quadric $Q'$ of rank 8 given by $w'^*(w')=0$. The variety $G\cap Q'$ is then 
defined by the equations $B'\wedge B'=0$, $B'\wedge w^*=0$, and $w'^*(w')=0$.

We now clearly see that by choosing an element in $\det U$ giving us an isomorphism $\bigwedge^2 U^*\rightarrow \bigwedge^2 U$ we get the desired isomorphism 
between the projection of any singular hyperplane section of $OG(5,10)$ to the intersection of the cone over $G(2,5)$ with the quadric $Q'$ defined above. We now 
need only to observe that the projection of any nodal codimension 5 section from its node is a generic section of the variety obtained above. It follows that it is 
isomorphic to the intersection of $G(2,V_5)$ with a quadric $Q$ containing a linear space $L_5$ of dimension 5 meeting the Grassmannian $G(2,5)$ in some four-dimensional 
quadric corresponding to some $G(2,V_4)$ for $V_4\subset V_5$ a 4 dimensional vector subspace of $V_5$.

For the converse let $G(2,V_5)\cap Q$ be an intersection containing a four-dimensional quadric $Q_4$. Then $Q_4$ must be equal to $G(2,V_4)\subset G(2,V_5)$ for $V_4\subset V_5$ a 4 dimensional vector subspace of $V_5$,
and then by adding Plucker quadrics defining the Grassmannian we can change $Q$ to $\tilde{Q}$ in such a way that $\tilde{Q}\cap G(2,V_5)=Q\cap G(2,V_5)$ and $\tilde{Q}$ contains 
the whole 5-dimensional space spanned by $Q$. It is now easy to see that such an intersection can be embedded as a linear section of $G\cap Q'$ hence is a projection 
of a singular section of $OG(5,10)$. Taking a generic intersection  $G(2,V_5)\cap Q$ containing a four-dimensional quadric $Q_4$ ensures us that the latter projection will
be performed from a node.

\end{proof}

We proceed similarly with Theorem \ref{z G(2,6) w S10}. The argument is due to L. Manivel.
\begin{proof}
Consider a point $p\in G(2,6)$. A Levi subgroup of its stabilizer is $GL(2)\times
GL(4)$ and the representation involved is 
$\bigwedge^2 V_2 \oplus (V_2\times V_4) \oplus \bigwedge^2 V_4$. Denote the
corresponding coordinates $(p,v,A)$. The tangent space to $G(2,6)$ in the fixed 
point is given by $A=0$. Hence choosing a hyperplane tangent to $G(2,6)$ at $p$ relies on choosing an element
$\omega\in \bigwedge^2 V_4^{*}$. Now the representation of $Sp(2)$ on the
ambient space of the image of the projection is $V_4\oplus V_4\oplus\bigwedge^{<2>} V_4$.

Observe that we recover the same representation on the space spanned by the
following section of $OG(5,10)$.	
As in the previous theorem consider $OG(5,10)$ as invariant under the action of
$GL(5)$ in the representation $\mathbb{P}(\det(V^*_5)\oplus \bigwedge^2
V^*_5\oplus V_5)$, take a decomposition of $V_5=V_4 \oplus \mathbb{C}$. The
corresponding representation of $GL(4)$ is $\mathbb{P}(\det(V^*_4)\oplus
\bigwedge^2 V^*_4 \oplus V^*_4 \oplus V_4 \oplus \mathbb{C})$. If we now fix a
symplectic form $\omega'$ on $V_4$ we get a representation of $Sp(2)$ given by
$V_4\oplus V_4\oplus\bigwedge^{<2>} V_4 \oplus \det V_4\oplus 2 \mathbb{C}$ we clearly
recognize the same representation as above as a codimension 3 component. Denote
its corresponding coordinates by $(w_1,w_2,B)$.

Let us now compare the equations defining the corresponding varieties. For the
projection of $G(2,6)$ from $p$ we have $A\wedge A=0$, $A\wedge v_1=0$, $A\wedge v_2=0$,
$\omega(v_1\wedge v_2)=0$. For the section of $OG(5,10)$ we have 
$B\wedge w_1=0$, $B(w_2)=0$, $B\wedge B=0$, $w_1(w_2)=0$. It is clear that these varieties are isomorphic.

For the converse observe that in the $GL(4)$ representation $\mathbb{P}(\det(V^*_4)\oplus
\bigwedge^2 V^*_4 \oplus V^*_4 \oplus V_4 \oplus \mathbb{C})$ the space $\mathbb{P}(V_4\oplus V_4^*)$ is a $\mathbb{P}^7$ meeting 
$OG(5,10)$ in a quadric of dimension 6 which by homogeneity is any such quadric. We easily check that the general codimension 3 linear section containing $\mathbb{P}(V_4\oplus V_4^*)$ 
is linearly isomorphic to the section by the space $V_4\oplus V_4\oplus\bigwedge^{<2>} V_4$ and conclude as above. 
\end{proof}

A similar procedure gives another proof of theorem \ref{z LG(3,6) w G(2,6)} additionally to the proof given in \cite{RI}.
\begin{proof}
Fix a point $p$ on the Lagrangian Grassmannian $LG(3,6)$. The stabilizer of $p$
contains $GL(3)$ as a Levi subgroup. The corresponding representation is:

$$\mathbb{P}(\det V_3 \oplus (S^2 V_3^{*} \otimes \det V_3 ) \oplus (S^2 V_3 \otimes \det
V_3^{*} ) \oplus \det V_3^{*} ),$$ 
denote the elements of the decomposition by
$(x,A,B,z)$. We know that the tangent space $T_p$ is the subspace given by $B=0, z=0$. Next, we observe that a choice of
hyperplane containing $T_p$ is equivalent to choosing an element $Q$ of
$\mathbb{P}((S^2 V_3\otimes \det V_3^{*}\oplus \det V_3^*)^{*})$. Consider the isotropy subgroup of
$GL(3)$ fixing $Q$. For a generic choice of $Q$ it contains $SL(2)$ in such a way that the corresponding $SL(2)$
representation on $V_3$ is $S^2 V_2$. The $SL(2)$ representation on the ambient space of $LG(3,6)$ is then
 $$\det(S^2 V_2) + S^2(S^2 V_2)\otimes \det(S^2 V_2^*) + S^2(S^2 V_2^*) \otimes \det(S^2 V_2) + \det (S^2 V_2^*).$$
Now using
$$S^2(S^2 V_2)=S^4 V_2 \oplus \mathbb{C}$$
and 
$$\det(S^2 V_2)=\mathbb{C},$$
and $V_2=V_2^*$, we get the representation on the ambient space of $LG(3,6)$ as
$$\mathbb{C}\oplus S^4 V_2\oplus \mathbb{C} \oplus S^4 V_2 \oplus \mathbb{C} \oplus \mathbb{C}.$$
The hyperplane section is given by a linear form on the subspace generated by the two last $\mathbb{C} $.
We hence get the following representation on the hyperplane section:
$$\mathbb{C}\oplus S^4 V_2\oplus \mathbb{C} \oplus S^4 V_2 \oplus \mathbb{C} .$$
Now the representation on the projection of the
hyperplane from $p$ is  
$$S^4 V_2\oplus \mathbb{C} \oplus S^4 V_2 \oplus \mathbb{C}.$$

On the other hand consider a $G(2,4)$ in $G(2,6)$ together with a decomposition
$V_6=V_4\oplus V_2$ and the associated $GL(2)\times GL(4)$ representation:
$\bigwedge^2 V_2 \oplus (V_2\otimes V_4) \oplus \bigwedge^2 V_4$. Next, consider
a general codimension 3 linear space $H_3$ in $V_2\otimes V_4$. We observe that
there is a subgroup of $GL(2)\times GL(4)$ which is isomorphic to $SL(2)$ which
fixes $H_3$. The associated $SL(2)$ representation on $V_6$ is $V_2\oplus S^3 V_2$ 
and hence on $\wedge^2 V_6$ is 
$$\mathbb{C}\oplus V_2\otimes S^3 V_2 \oplus \bigwedge^2 (S^3 V_2).$$
We now observe that 
$$V^*_2\otimes S^3 V_2= S^4 V_2\oplus S^2 V_2$$ 
and 
$$\bigwedge^2 (S^3 V_2)= S^4 V_2 \oplus \mathbb{C}.$$ 
Hence we get 
$$\wedge^2 V_6=\mathbb{C}\oplus S^4 V_2\oplus S^2 V_2 \oplus S^4 V_2 \oplus \mathbb{C}$$
then $H_3$ is given by considering the subspace with the projection onto $S^2 V_2$ equal to zero. 
Then on $H_3$ we have 
$$\mathbb{C}\oplus S^4 V_2 \oplus S^4 V_2 \oplus \mathbb{C}.$$

The comparison of the equations describing our varieties in the coordinates corresponding
to the decompositions above gives us an isomorphism between any nodal hyperplane section of $LG(3,V_6)$ and a generic codimension 3 section of 
$G(2,6)$ containing some $G(2,4)$ i.e. a generic quadric $Q\subset G(2,6)$ of dimension 4. 
Indeed in both cases we have three types of equations:
\begin{itemize}
 \item the first type of equations is described by two maps: 
$$\pi_1 \colon S^2(S^4 V_2)=S^8 V_2\oplus S^4 V_2 \oplus \mathbb{C}\rightarrow S^4 V_2 \oplus \mathbb{C}$$
the canonical projection, and 
$$ m_1\colon (\mathbb{C} \oplus S^4 V_2)\otimes \mathbb{C}\rightarrow S^4 V_2 \oplus \mathbb{C}$$
standard multiplication.
If we now have $(x,A,y,B)\in S^4 V_2\oplus \mathbb{C} \oplus S^4 V_2 \oplus \mathbb{C}$ the equations correspond to 
$\pi_1(A)-m_1(x\otimes (y,B))=0$.
\item the second type of equations correspond to a subset of $S^4 V_2\otimes S^4 V_2\oplus \mathbb{C}$. We saw that $S^4 V_2\oplus \mathbb{C}$ 
may be seen as a subspace of $S^2((S^4 V_2))$. The subset is just the syzygies between the quadrics. More precisely we have a map:
    $$S^4 V_2 \otimes S^2(S^4 V_2)\rightarrow S^3(S^4 V_2)$$
given by multiplication. This maps induces a map $S^4 V_2\otimes S^4 V_2\oplus \mathbb{C} \rightarrow S^3(S^4 V_2)$ which defines our second type of equations.

\item the last type consists of a single equation. It corresponds to a map $S^2(S^4 V_2 \oplus \mathbb{C})\rightarrow \mathbb{C}$ induced by the composition 
$$S^2(S^4 V_2 \oplus \mathbb{C})\rightarrow S^2(S^2 (S^4 V_2))\rightarrow S^4(S^4 V_2)\rightarrow \mathbb{C}.$$
\end{itemize}

\end{proof}
The situation in theorem \ref{z G2 w LG(3,6)} is more complicated.
Indeed as the involved representation of $\mathbb{G}_2$ is the adjoint one and the general singular section appears 
on an orbit of codimension 1 we have a two-dimensional subgroup of $\mathbb{G}_2$ acting on the general singular hyperplane section.
We can check easily that its Lie algebra decompose as semi-direct product of $\mathfrak{sl}_1$ and a nilpotent algebra. Hence we are left with the comparison of representations of $\mathfrak{sl}_1$
which will not help much in finding and understanding the isomorphism.

The theorem \ref{z G2 w LG(3,6)} is however still true and will be proved by guessing the isomorphism for one representative of each orbit of the dual variety 
giving a nodal section of $G_2$. 
We are aware that this does not shed light on the geometry of the construction but we believe that the theorem itself has interesting geometric consequences.

We shall compare nodal hyperplane sections of $G_2$ with codimension 2 sections of $LG(3,6)$ containing a quadric.

Observe that from the description of $G_2$ contained in subsection \ref{eqG2}, we obtain an explicit linear embedding of $G_2$ in $G(2,7)$. Using it we can 
explicitly write down a set of generators of the ideal of $G_2$ in coordinates $(a,\dots, n)$ of $\mathbb{P}^{13}$ by listing Pfaffians of the following matrix:
\begin{displaymath}
M_{G_2}=\left( \begin{array}{ccccccc}
0&-f&e&g&h&i&a\\
f&0&-d&j&k&l&b\\
-e&d&0&m&n&-g-k&c\\
-g&-j&-m&0&c&-b&d\\
-h&-k&-n&-c&0&a&e\\
-i&-l&g+k&b&-a&0&f\\
-a&-b&-c&-d&-e&-f&0
\end{array}\right),
\end{displaymath}
where $h,(m,i,a,e),(c,f,g+k),g,(b,d,n,l),j$ corresponds to the decomposition 
$$\operatorname{Ad}_{\mathbb{G}_2}=\det V_2 \oplus (S^3 V_2\otimes \det V_2^*) \oplus (S^2V_2\otimes \det V_2^*)\oplus \mathbb{C} \oplus 
(S^3 V_2^*\otimes \det V_2)\oplus \det V_2^*.$$

On the other hand if we choose coordinates for $S^2 V_3\otimes \det(V_3^*)$ and $S^2 V_3^*\otimes \det(V_3)$ as well as $\det V_3$ and $\det V_3^*$ we 
can also list the considered equations defining $LG(3,6)$ as follows:
$$A\cdot B=x\cdot y\cdot \operatorname{id},
\wedge^2 A=B\cdot x,
\wedge^2 B=A\cdot y,$$
where $A$ and $B$ are the symmetric matrices given by the chosen coordinates.

\begin{proof} Passing to the proof we check that we have exactly two orbits of nodal hyperplane sections of $G_2$. 
By lemma \ref{orbity} we can choose representing sections as follows 
\begin{itemize}
 \item The section given by $j=c+f$ is a representative of the open orbit of the dual variety. It is singular at $h=1$. 
\item The section given by $f=m$ is also nodal at $h=1$ but corresponds to a hyperplane represented in the dual space by a point which lies in the intersection
of the dual variety with the quadric defined by the Killing form i.e. is an element of the 11-dimensional orbit . 
\end{itemize}

Consider first this special case. Then the image of the
projection is given by all Pfaffians of the
matrix $M_{G_2}$ not involving $i$ and the quadric $a^2+ng-e(c+i)$. We check
directly that these equations define the same
ideal as the linear section of $LG(3,6)$ with the subspace of codimension 2
parametrized via

$$\begin{aligned}
(&u, x_{1,1},x_{1,2},x_{1,3},x_{2,2},x_{2,3},x_{3,3},
y_{1,1},y_{1,2},y_{1,3},y_{2,2},y_{2,3},y_{3,3}, z)=\\
(&n,f, c,-g-k,e,a,c+i,g,-d,-f,l,-b,-d,j)
  \end{aligned}$$

In other terms the equations of the projection correspond to the data
\begin{displaymath}
n, \left(\begin{array}{ccc}
 -f& c&-g-k\\
c&e&a\\
-g-k&a&c+i
\end{array}\right),
\left(\begin{array}{ccc}
 g&-d&-f\\
-d&l&-b\\
-f&-b&-d
\end{array}\right)
,j.
\end{displaymath}
In the example corresponding to the general case we project from $h=1$ and the
projection is given by the Pfaffians of
$M_{G_2}$ not involving $h$ and the quadric $a(a-n-d)+gk-e(i+e+b)=0$. We can
rearrange the equations to be defined by
the following data
$$a-d, \left(\begin{array}{ccc}
i+e+b&d&-g\\
d&-b&c\\
-g&c&-e-m\\

\end{array}\right),
\left(\begin{array}{ccc}
-d&-m&-c-f\\
-m&-a+n+d&g+k\\
-c-f&g+k&d-l
\end{array}\right)
,m+b.$$
being isomorphic to a codimension 2 linear section of the Lagrangian Grassmannian containing the 3 dimensional quadric
defined by $m=b=c=d=f=k=0$ and $g^2+(i+e+b)(e+m)+(a-d)(-a+n+d)$.
 
The equality of the above ideals can be checked by hand, writing each equation of one variety as a linear 
combination of equations defining the other. To save space we wont write down all the equations here, instead 
we provide simple Macaulay scripts performing the computation.
For the $f=m$ section of $G_2$:
\begin{lstlisting}
R=QQ[a,b,c,d,e,f,g,h,i,j,k,l,m,n];
M=matrix{{0,-h,-n,e,a,-c,-k},{h,0,e,a,i,g,-f},{n,-e,0,c,-g-k,m,d},
{-e,-a,-c,0,-f,-d,-b},{-a,-i,g+k,f,0,b,-l},
{c,-g,-m,d,-b,0,-j},{k,f,-d,b,l,j,0}};
G2=pfaffians(4,M);
H=G2+ideal(m-f);
projection=eliminate(h,H);
S=QQ[a,b,c,d,e,f,g,i,j,k,l,n];
ProjectionG2=(map(S,R,{a,b,c,d,e,f,g,0,i,j,k,l,f,n}))(projection);
x=n;
A=matrix{{-f,c,-g-k},{c,e,a},{-g-k,a,c+i}};
B=matrix{{g,-d,-f},{-d,l,-b},{-f,-b,-d}};
z=j;
adjugate= MM -> matrix( for ii from  0 to 2 list 
(for jj from  0 to 2 list 
(-1)^(ii+jj)*det(submatrix'(MM,{ii},{jj})) ));
T=adjugate(A);
U=adjugate(B);
LG=ideal(A*B-x*z,T-x*B,U-z*A);
LG == ProjectionG2
\end{lstlisting}
and for the $j=c+f$ section:
\begin{lstlisting}
R=QQ[a,b,c,d,e,f,g,h,i,j,k,l,m,n];
M=matrix{{0,-h,-n,e,a,-c,-k},{h,0,e,a,i,g,-f},{n,-e,0,c,-g-k,m,d},
{-e,-a,-c,0,-f,-d,-b},{-a,-i,g+k,f,0,b,-l},
{c,-g,-m,d,-b,0,-j},{k,f,-d,b,l,j,0}};
G2=pfaffians(4,M);
H=G2+ideal(j-c-f);
pr=eliminate(h,H);
S=QQ[a,b,c,d,e,f,g,i,k,l,m,n];
ProjectionG2=(map(S,R,{a,b,c,d,e,f,g,0,i,c+f,k,l,m,n}))(pr);
x=a-d;
A=matrix{{i+e+b,d,-g},{d,-b,c},{-g,c,-e-m}};
B=matrix{{-d,-m,-c-f},{-m,-a+n+d,g+k},{-c-f,g+k,d-l}};
z=m+b;
adjugate= MM -> matrix( for ii from  0 to 2 list 
(for jj from  0 to 2 list 
(-1)^(ii+jj)*det(submatrix'(MM,{ii},{jj})) ));
T=adjugate(A);
U=adjugate(B);
LG=ideal(A*B-x*z,T-x*B,U-z*A);
LG == ProjectionG2
\end{lstlisting}
For the other direction we observed that all maximal dimensional quadrics in $LG(3,6)$ are equivalent by the action of the symplectic group. We can also observe that two general codimension 2 sections containing 
a fixed quadric are linearly isomorphic. Indeed, we prove by a straightforward computation in Macaulay 2 on the corresponding Lie algebra representation that the dimension of the stabilizer of the line orthogonal in $\bigwedge^{<3>} V_6$ to 
the codimension 2 space in $\bigwedge^{<3>} V_6$ corresponding to the general orbit is 2. It follows that the orbit of this codimension 2 space is open and dense in the variety of all codimension
2 sections containing a quadric. 
\end{proof}

\end{section}
\begin{section}{Geometric transitions}\label{mir}
The main results of the paper provide a very concrete and geometrically simple
way to connect different families of Calabi-Yau threefolds with Picard 
number one from the list of Borcea as well as the complete list of prime Fano
threefolds of index 1 and genus $\leq 10$.
In the context of Calabi-Yau threefolds such a connection can be formulated as a
composition of two so-called geometric transitions.

In the theory of Landau Ginzburg models for Fano threefolds these are the
simplest possible way to connect two different Fano threefolds such that one can
hope to keep track of the Landau Ginzburg Models involved.
For the sake of future applications in this context we present in this section
further results concerning the geometry of the construction involved.

Let $F_g$ be a general Fano threefold of genus $g$ in its anti-canonical
embedding. By the theorem of Mukai $F_g$ is a transverse linear section of
$M_g$. It is hence clear that $F_g$ admits a flat deformation to a 
nodal Fano threefold $F'_g$ being a transverse linear section of a nodal linear
section $L$ of $M_g$ studied in this paper. 

\begin{lemm} \label{bir} The projections considered in this paper are birational
maps.
\end{lemm}
\begin{proof}
From the fact that Mukai varieties are generated by quadrics and are not cones we know that their projection 
from a point lying on them is a birational morphism contracting the tangent cone to its base. To conclude we 
need only to observe that the considered projections are restrictions of these projections to nodal hyperplane sections which are not cones.
\end{proof}

By our main results and Lemma \ref{bir} the projection of $F'_g$ from its node
is a possibly singular Fano threefold 
$\hat{F}_{g-1}$ containing a quadric surface. Using the Mukai theorem again we
get a flat deformation of $\hat{F}_{g-1}$ to a general Fano threefold of genus
$g-1$. To complete the picture we need only to understand the singularities of
$\hat{F}_{g-1}$.

\begin{prop}
If $F_g$ is a three-dimensional nodal proper linear section of $M_g$ then the
only singularities of its projection $\hat{F}_{g-1}$ are again nodes. Moreover
the number of nodes on $\hat{F}_{g-1}$ is $5,4,4,3$ for $g=7,8,9,10$
respectively.  
\end{prop}
\begin{proof}
Observe that the by Lemma \ref{bir} the singularities of $\hat{F}_{g-1}$ are exactly the images of the lines contracted by the projection.
Moreover all of them lie on the smooth quadric being the exceptional locus of the projection. We see that the singularities are isolated and their number  
is the number of lines contracted by each projection. To compute it observe that each line contracted is contained in 
the intersection of $F_g\cap T_pM_g$ where $p$ is the center of projection. We now observe that $M_g \cap T_pM_g$ is one of the following:
\begin{itemize}
 \item a cone over a Grassmannian $G(2,5)$ for $g=7$
 \item a cone over a product $\mathbb{P}^1\times \mathbb{P}^3$ for $g=8$
 \item a cone over a Veronese surface for $g=9$
 \item a cone over a twisted cubic for $g=10$.
\end{itemize}
Now we observe that it is always a variety of codimension 3 in $T_p M_g$, hence $F_g\cap T_pM_g$ is a union of as many lines as the degree of corresponding cone
i.e. $5,4,4,3$ for $g=7,8,9,10$
respectively.
 
Finally we claim that all singularities are ordinary double points. Since we have only a few cases to consider one can check every 
singularity of a representative of each orbit of varieties and check their type of singularities on the computer. We shall however use a more general argument.
We just observe that the variety $\hat{F}_{g-1}$ is a general proper linear section of $M_{g-1}$ containing a chosen quadric surface. As the quadric surface is 
a scheme theoretical proper linear section of $M_{g-1}$ one can use the following Proposition which is a reformulation of \cite[theorem 2.1]{GPrim2} in a slightly 
more general context.
\begin{prop}
Let $X$ be a smooth projective variety of dimension $s+2$. Let $S\subset X$ be a smooth codimension $s$ surface being a scheme theoretical base locus 
of a linear subsystem $\mathcal{L}\subset |\mathcal{O}_X(d)|$, for some $d\in \mathbb{N}_{\geq 1}$. Then the intersection of a set of $s-2$ generic divisors 
from the system $\mathcal{L}$ is a threefold with only ordinary double points as singularities.  
\end{prop}
\begin{proof}
 The proof of \cite[theorem 2.1]{GPrim2} can be reproduced without changes. 
\end{proof}
\end{proof}

\end{section}
\begin{section}{Applications to moduli}\label{modBNL}
In this section we explain how the results of the paper may be used to describe
some moduli spaces of bundles on curves,
K3 surfaces or Fano 3-folds. We follow the idea presented in \cite{RI}. More
precisely, let $X$ be a generic curve of genus
$g\leq 9$ (or K3 surface, Fano 3 fold of genus $g\leq 10$). Then $X$ is a linear
section of the variety $M_g$ and each nodal
linear space $H$ containing $X$ with node outside $X$ induces an embedding
$\pi_H: X\rightarrow M_{g-1}$.
\begin{prop}\label{different embeddings}
Two different linear spaces $H$ of maximal dimension define different
embeddings.
\end{prop}
\begin{proof} The proof contained in \cite{RI} is valid in all cases.
Indeed, assume we have $H_1$ and $H_2$ inducing isomorphic embeddings. Then the
intersections of the linear spans of their
images with $M_{g-1}$ are projectively isomorphic. Hence their preimages by the
projections are also projectively isomorphic,
but these are just the sections of $M_g$ by linear spans of $X$ and the
projecting nodes. By Mukai's result (see \cite{Muk1})
these are isomorphic only if they lie in the same orbit of the appropriate group
action. This would induce an automorphism of
$X$ which by genericity would have to be trivial and in consequence the
automorphism of $M_g$ inducing it would also need to
be trivial.
\end{proof}
Let us now make a side remark.
\begin{rem}\label{uwaga}
The results of this paper may be used to describe the space of polarized K3
surfaces of genus $g-1$ containing a chosen
general curve o genus $g\leq 9$ in its canonical embedding. One might observe
that such a general K3 surface $S$ has Picard
number 2 generated by the hyperplane $H$ and a conic $Q$. The curve of genus $g$
lies then in the system $|H+Q|$, which
induces a morphism from $S$ to $M_g$ contracting $Q$ to a node. Its image is a
nodal linear section of $M_g$ hence
is a linear section of a nodal hyperplane section. It follows that it appears in
our construction. Concluding, the space
of K3 surfaces genus $g-1$ containing a chosen general curve o genus $g\leq 9$
in its canonical embedding is isomorphic to
a linear section of the dual variety to $M_g$ by the dual linear space to the
span of the chosen curve. The same is valid
for Fano threefolds containing a K3 surface of genus smaller by one.
\end{rem}

Let us now consider the situation for $g=7$.
\begin{lemm}\label{stable g=7}
Let $X$ be a general curve of genus $g=7$ embedded as a linear section of
$OG(5,10)$. Consider a nodal linear space $H$
of dimension 7 containing $X$ and denote by $p$ its node. Let $\pi: X\rightarrow
G(2,5)\cap Q$ be the projection from $p$
and let $E$ be the pullback of the universal quotient bundle on $G(2,5)$. Then
$E$ is stable on $X$.
\end{lemm}
\begin{proof} By assumption $X$ has no $g_4^1$. Observe first that by
construction $E$ is a rank 2 bundle on $X$ such that $c_1(E)=K_X$, $h^0(E)=5$.
By contradiction assume that there is a line bundle $L$ on $X$ of degree $d\geq
4$ which is a sub-bundle of $E$.
We get the exact sequence:
$$ 0\longrightarrow L\longrightarrow E\longrightarrow M \longrightarrow 0,$$
where $M=K_X\otimes L^{*}$. It follows that $h^0(L)+h^0(M)\geq h^0(E)=5$. On the
other hand by the Riemann-Roch formula we have $h^0(L)-h^0(M)= d-6\geq-2$. Hence
$h^0(L)\geq 2$, which contradicts the assumption on $C$.
\end{proof}
The following well known corollaries (see \cite{MukBN},\cite{Mukdual})follow.
\begin{cor} Let $C$ be a general curve of genus 7. Then, the Brill Noether locus
$M_C(2,K_X,5)$
parametrization rank 2 vector bundles with canonical determinant and with 5
global
sections is a Fano
variety of genus 7 obtained as the dual linear section of $OG(5,10)$.
\end{cor}
\begin{proof}
By dimension count it is a direct consequence of Proposition \ref{different
embeddings},
Lemma \ref{stable g=7} and the fact that the orthogonal Grassmannian is
self-dual.
\end{proof}
We also recover in the same way the result concerning K3 surfaces.
\begin{cor} Let $(S,L)$ be a BN-general polarized K3 surface of genus 7. Then,
the K3 surface $M=M_C(2,K_X,3)$
parametrization rank 2 vector bundles with canonical determinant and second Chern
class of degree 3 is isomorphic
to the dual linear section of $OG(5,10)$.
\end{cor}

Let us pass to the case $g=10$
\begin{lemm}\label{stable g=10}
Let $X$ be a curve of genus $g=10$ embedded as a general linear section of
$G_2$. Consider a nodal linear
space $H$ of dimension 10 containing $X$ and denote by $p$ its node. Let $\pi:
X\rightarrow LG(3,V)$, where $V$ is a
vector space of dimension 6, be the projection from $p$ and let $E$ be the
pullback of the universal quotient bundle on
$G(3,V)$. Then $E$ is stable on $X$.
\end{lemm}
\begin{proof} The proof is completely analogous to \cite[Prop. 5]{M3l3}
\end{proof}
\begin{rem}
The bundles considered in Lemma \ref{stable g=10} were already constructed in a
different context by Kuznetsov in \cite{Kuz}.
\end{rem}

\begin{cor} \label{BN on curve of genus 10} Let $C$ be a general curve of genus
10 contained in a K3 surface.
Then, the non-abelian Brill-Noether locus $M(3,K_C,6)$ contains a surface $M_L$
of degree 6 which is the linear section
of the dual variety of $G_2$ by the dual space to the span of $C$. The
bundles corresponding to points of this
surface correspond to embeddings of $C$ into $LG(3,6)$.
\end{cor}
\begin{proof}
It follows from the fact that the variety dual to $G_2$ is a hypersurface of
degree 6 (see \cite{Hol}) and the fact that our
$G_2$ is the homogeneous space of the adjoint representation of the simple Lie
group $\mathbb{G}_2$. Repeating the reasoning
contained in Remark \ref{uwaga} we prove that in this construction we obtain all
embeddings of $C$ into $LG(3,6)$.
\end{proof}

We can view the surface as a non-abelian Brill-Noether locus of all quotient
Lagrangian bundles, but first we need to
prove that such moduli space indeed exists. Let us make a general argument.
Let $C$ be a curve of genus $g$. Let $E=\mathcal{O}_C^{2n}$, we say that a
quotient bundle $L$ of $E$ of rank $n$ on $C$
with $c_1(L)=K_C$ is Lagrangian if the image of the map associated to this
bundle to the Grassmannian $G(n,2n)$ is contained
in some $LG(n,2n)$. This is equivalent to saying that there is an exact
sequence:
$$0\longrightarrow L^{*}\longrightarrow E \longrightarrow L \longrightarrow 0$$
and a nondegenerate two form $\omega$ on $H^0(E)$ such that all fibers of
$L^{*}$ are isotropic with respect to $\omega$.
It follows that for a Lagrangian bundle the kernel of the map
$\phi_L\colon H^0(\bigwedge^2(E))\longrightarrow H^0(\bigwedge^2(L))$ is
non-empty. The set of bundles with this property
may be given the structure of a scheme as in \cite{MukBN} because it may be
regarded as the degeneracy locus of a map of
bundles. This scheme may be saturated with respect to the rank of the map. Now
the set of Lagrangian bundles is an open
subset of any component of the scheme of bundles for which $\phi$ has non-empty
kernel of chosen dimension. It follows that
there is a natural structure of scheme on the set of Lagrangian bundles. If $C$
is a curve of genus $10$ and $n=3$ we have
$h^0(C,\bigwedge^2(E))=h^0(C,\bigwedge^2(L))=15$. Hence our moduli is expected
to be a hypersurface in the Brill-Noether
locus $M(3,K_C,6)$ and in fact it is such by Corollary \ref{BN on curve of genus
10}.

\begin{cor} Let $(X,L)$ be a polarized K3 surface of genus 10.
The space $M_X(3,L,3)$ contains a smooth plane curve of degree 6 which is the
linear section of the dual variety of
$G_2$ by the dual space to the span of $X$. The bundles corresponding
to points of this curve correspond to
embeddings of $X$ into $LG(3,6)$.
\end{cor}
\begin{proof}
The only thing that remains to be proved is the smoothness of the sextic. It 
follows from genericity and the fact that the singular set of the variety dual to $G_2$ is of codimension 2.
\end{proof}

\begin{cor} The moduli space $M_X(3,L,3)$ is a smooth K3 surface isomorphic to a
double cover of $\mathbb{P}^2$ branched over a sextic.
\end{cor}
\begin{proof} By \cite{MukK3} the moduli space $M_X(3,L,3)$ is a smooth K3
surface. Let $C$ be the plane sextic
curve on the K3 surface $M_X(3,L,3)$. Then $C$ is of genus 10 and admits a
$g^2_6$. By Green and Lazarsfeld theorem
the $g^2_6$ extends to the K3 surface giving a map to $\mathbb{P}^2$.
\end{proof}
For a more detailed description of the above Moduli space see \cite{M3l3}.
\begin{rem}
By considering compositions of the studied projections we obtain different
subvarieties of suitable Moduli spaces or
Brill-Noether loci.
\end{rem}
\section*{Acknowledgments}
The work was initiated during the author's stay at the University of Oslo between
03.2009 and 03.2010. I would like to especially
thank Kristian Ranestad for suggesting me the subject, enlightening discussions
and wonderful hospitality. I acknowledge also
G. Kapustka, J. Buczy\'nski, J. Weyman, L. Manivel and A. Kuznetsov for remarks
and advices. Finally I would like to express my gratitude to the referee who made many crucial suggestions.

\end{section}

\end{document}